\documentclass[]{article}

\usepackage[leqno]{amsmath}
\usepackage{amssymb,amsfonts,xfrac,MnSymbol}
\usepackage{amsthm}
\usepackage{cite}
\usepackage{hyperref}
\usepackage{todonotes}
\usepackage{enumitem}
\usepackage{graphicx}
\usepackage{tikz}
\usepackage{color}
\usepackage{import}
\usepackage{mathtools} 

\numberwithin{equation}{section}

\newtheorem{theorem}{Theorem}[section]
\newtheorem{claim}[theorem]{Claim}
\newtheorem{proposition}[theorem]{Proposition}
\newtheorem{lemma}[theorem]{Lemma}
\newtheorem{corollary}[theorem]{Corollary}

\newtheorem*{theorem*}{Theorem}
\newtheorem*{claim*}{Claim}
\newtheorem*{proposition*}{Proposition}
\newtheorem*{lemma*}{Lemma}
\newtheorem*{corollary*}{Corollary}

\newtheorem{theoremA}{Theorem}

\theoremstyle{definition}
\newtheorem{definition}[theorem]{Definition}

\newtheorem{question}[theorem]{Question}

\newtheorem{fact}[theorem]{Fact}

\newtheorem*{definition*}{Definition}
\newtheorem*{observation*}{Observation}
\newtheorem*{remark*}{Remark}
\newtheorem*{example*}{Example}
\newtheorem*{question*}{Question}
\newtheorem*{exercise*}{Exercise}
\newtheorem*{fact*}{Fact}
\newtheorem*{notation*}{Notation}

\newcommand{\bbH}{\mathbb{H}}

\newcommand{\bbN}{\mathbb{N}}

\newcommand{\bbR}{\mathbb{R}}

\newcommand{\calW}{\mathcal{W}}

\newcommand{\ii}{^{-1}}

\newcommand{\tild}[1]{\widetilde{#1}}

\DeclareMathOperator{\sep}{sep}

\DeclareMathOperator{\vol}{vol}
\DeclareMathOperator{\cut}{cut}
\DeclareMathOperator{\length}{length}
\DeclareMathOperator{\im}{Im}

\title{Hyperbolic groups with logarithmic separation profile}
\author{Nir Lazarovich\thanks{Supported by the Israel Science Foundation (grant no. 1562/19), and by the German-Israeli Foundation for Scientific Research and Development.}~ and Corentin Le Coz\thanks{Supported by the Israel Science Foundation (grant no. 2919/19)}}
\date{}

\begin{document}

\maketitle

\begin{abstract}
   We prove that hyperbolic groups with logarithmic separation profiles split over cyclic groups.
   This shows that such groups can be inductively built from Fuchsian groups and free groups by amalgamations and HNN extensions over finite or virtually cyclic groups.
   However, we show that not all groups admitting such a hierarchy have logarithmic separation profile by providing an example of a surface amalgam over a cyclic group with superlogarithmic separation profile.
\end{abstract}

\section{Introduction}

The separation profile was first introduced by Benjamini-Schramm-Tim\'ar~\cite{benjamini2010separation} in 2012.
It measures large scale connectivity of infinite graphs, in the spirit of the celebrated theorem of Lipton and Trajan for planar graphs \cite{liptontarjan1}.

\begin{definition}[Benjamini-Schramm-Tim\'ar \cite{benjamini2010separation}]
	Given a finite graph $\Gamma = (V\Gamma,E\Gamma)$, we shall say that a set of vertices $C \subset V\Gamma$ \textbf{cuts} (or \textbf{separates}) the graph $\Gamma$ if the connected components of the subgraph induced by $V\Gamma - C$ contain at most $\frac 12 |V\Gamma|$ vertices.
	
	We define the \textbf{cut} of the graph $\Gamma$, denoted $\cut \Gamma$, as the minimal size of a separating set.
	
	We define the \textbf{separation profile} of a bounded degree infinite graph $G$ as the following non-decreasing function from $\bbN^*$ to $\bbN^*$:
$$ \sep_G(n) = \sup_{\Gamma \subset G, |V\Gamma| \le n} \cut\Gamma $$
\end{definition}

We shall consider such function endowed with the partial order defined by $g \preceq h$ if and only if there exists $D>0$ such that $g(n) \le Dh(Dn) + D$ for any $n \ge D$.
We denote by $\asymp$ and $\prec$ the associated equivalence relation and strict partial order, respectively.

As noticed in \cite{benjamini2010separation}, the factor $1/2$ does not play an important role in the previous definition.
Replacing it by any $\beta \in (0,1)$ would give an equivalent profile.\\

The separation profile is a coarse-geometric monotone invariant (see Proposition \ref{prop: sep monotone under coarse embeddings}).
To our knowledge, the only such invariants that were previously defined are volume growth and asymptotic dimension \cite{gromov1993asymptotic}.
The separation profile is a much finer invariant and has been generalized by Hume-Mackay-Tessera \cite{hume2020poincare,hume2021poincare} into a spectrum of profiles called Poincar\'e profiles.
For a survey on this topic, we refer to the first part of the thesis of the second author \cite{lecoz2020thesis}.

It is proved in \cite{benjamini2010separation} (see also \cite{hume2020poorly}) that if a hyperbolic group has $\sep_G(n) \prec \log(n)$ then $\sep_G(n)$ is bounded and $G$ is virtually free.

In this paper we investigate the smallest possible non-virtually free case, namely $\sep_G(n) \preceq \log(n)$.

\begin{theoremA}\label{thm: main thm}
    Let $G$ be a hyperbolic group with $\sep_G(n) \preceq \log(n)$, then $G$ is Fuchsian or splits over finite or virtually cyclic subgroups.
\end{theoremA}

This theorem is proved in Section \ref{section: proof of theorem A}, but let us give here a sketch of proof.
The first step consists in showing that the spheres of $G$ have bounded separating sets.
This is done by projecting the separating set of some suitable annulus.
Then, we make these sphere separating sets converge in $\partial G$.
This implies the existence of local cut points in $\partial G$, and the conclusion follows from Bowditch \cite{bowditch1998cut}.

\begin{corollary}
    Let $G$ be a hyperbolic group without 2-torsion. If $\sep_G(n) \preceq \log(n)$ then $G$ can be inductively built from Fuchsian groups and free groups by amalgamations and HNN extensions over finite or virtually cyclic groups.
\end{corollary}

\begin{proof}
We can apply Theorem \ref{thm: main thm} to $G$.
Either $G$ is Fuchsian and we are done. Otherwise, $G$ splits over virtually cyclic groups.
The edge groups are virtually cyclic, hence quasiconvex in $G$.
This implies that the vertex groups of this splitting are quasiconvex and hence hyperbolic.  
By the monotonicity of the separation profile, the separation profile of the vertex groups $H$ is $\sep_H(n)\preceq \sep_G(n)\asymp\log(n)$.
Therefore, we can successively apply Theorem \ref{thm: main thm} to split $G$ over virtually cyclic subgroups.
Using the Strong Accessibility by Louder-Touikan \cite{louder2017strong} this process terminates.
\end{proof}

A group with conformal dimension at least one always has a separation profile bounded below by $\log$, from \cite{hume2020poorly}. 
Using a recent result of Carrasco-Mackay~\cite{carrasco2020conformal} giving a characterization of hyperbolic groups with conformal dimension one, we get the following corollary.

\begin{corollary}\label{corollary: confdim>1 implies sep>log}
	Let~$G$ be a one-ended hyperbolic group with no~$2$-torsion.
	If the (Ahlfors regular) conformal dimension of~$G$ is strictly greater than~$1$, then its separation profile is strictly greater than~$\log$.
\end{corollary}

In this generality, to our knowledge this improves the previously known lower bounds.
We do not know if this is sharp.

Lower bounds on separation profiles can be obtained from Poincar\'e inequalities in the boundary at infinity of hyperbolic groups, see Hume-Mackay-Tessera~\cite[Theorem 13]{hume2020poincare}. Finding general Poincar\'e inequalities is an important challenge and this corollary can be seen as a step in this direction.

The following theorem shows that the converse of Theorem~\ref{thm: main thm} (and the subsequent corollaries) is false.

\begin{theoremA}\label{theorem: gluing tori along filling curves}
	Let~$S$ be the surface amalgam obtained by gluing two closed orientable hyperbolic surfaces along a closed filling curve in each.
	Then, $\sep_{\pi_1 S}(n)\succ \log(n)$.
\end{theoremA}
	
	From Carrasco-Mackay~\cite{carrasco2020conformal}, such a group has conformal dimension equal to~$1$.
	
	From \cite{hume2020poincare}, a hyperbolic group with conformal dimension one always have a separation profile bounded above by any $n^{\epsilon}$, with $\epsilon>0$.
	To our knowledge, this is this is the first example of such a group whose separation profile is not logarithmic.
	This implies in particular that the conformal dimension is not attained \cite[Theorem 11]{hume2020poincare}.
	
	We believe that when the curves are not filling, the separation profile is actually~$\log$.
	
\begin{question}
	Let~$S$ be a simple surface amalgam obtained by gluing two closed hyperbolic orientable surfaces along simple curves.
	Do we have~$\sep \pi_1 S \asymp \log$?
\end{question}

From Hume-Mackay-Tessera~\cite{hume2020poincare} study of relations between conformal dimension and separation profiles, we as well can formulate the following question:

\begin{question}
	If a hyperbolic group has a separation profile bounded above by~$n^{\epsilon}$ for every positive~$\epsilon$, does it imply that it has conformal dimension one?
\end{question}

%
%

{\it Acknowledgments:}
The authors would like to thank John Mackay and Ilya Gekhtman for interesting discussions.

\section{Proof of Theorem~\ref{thm: main thm}}\label{section: proof of theorem A}

Let $G$ be a one-ended hyperbolic group. 
By abuse of notation let us denote by $G$ also the Cayley graph of $G$ with respect to some fixed finite generating set, and assume it is $\delta$-hyperbolic. We denote by $o$ the identity element of $G$.
For every $R>0$, $B_R$ (resp. $S_R$) denote the ball (resp. the sphere) of radius $R$ centred at $o$.
 
\begin{definition}
For each $R>0$, let $\pi_R : G - B_R \to S_R$ be a projection defined by $\pi_R(y) = [o,y] \cap S_R$, where $[o,y]$ is some choice of geodesic joining $o$ and $y$.

For any $\alpha>0$, we call \textbf{$\alpha$-step path} any family of vertices $v_1,\ldots,v_k$ such that $d(v_i,v_{i+1}) \le \alpha$ for any $i=1,\ldots, k-1$.\end{definition}

\begin{fact}\label{fact: thickening of spheres are connected}
For each $R>0$, $\pi_R$ is $2\delta$-Lipschitz.
As a consequence, there is a constant $\delta'$ such that the $\delta'$-neighbourhood of any sphere in $G$ is connected.
Moreover, we can choose $\delta$ large enough so that we can take $\delta'=\delta$.
\end{fact}

\begin{proof}
The fact that the projection map $\pi_R$ is $2\delta$-Lipschitz is a straightforward consequence of the $\delta$-slimness of geodesic triangles in $G$.
Let us prove the second assertion, let us take $x,y \in S_R$.
From \cite{BestvinaMess1991}, there is some $C$ such that there exist $x',y' \in S_R$ such that $d(x,x'),d(y,y') \le C$ satisfying that $x'$ and $y'$ lie in some infinite connected component of $G - B_R$.
Since $G$ is one-ended, $G - B_R$ contains a single infinite connected component $D$, and thus $x',y' \in D$.
The vertices $x'$ and $y'$ can then be joined by a path in $G - B_R$.
We can project this path on $S_R$ using the map $\pi_R$.
Since $\pi_R$ is $2\delta$-Lipschitz, the projected path is a $2\delta$-step path in $S_R$.
It follows that the $\max(\delta,C)$-neighbourhood of any sphere of $G$ is connected.
\end{proof}

In the sequel, we will assume that $\delta$ is chosen large enough so that the spheres of $G$ are $\delta$-connected.
This motivates the following definition.

\begin{definition}
    We say that $G$ has \textbf{bounded spheres separation} if every $\delta$-neighbourhood of a sphere has a cut-set of uniformly bounded size.
\end{definition}

We are now able to state our key lemma.

\begin{lemma}\label{lem: log sep to bdd sphere sep}
    If $G$ is hyperbolic and $\sep_G(n) \preceq \log(n)$, then $G$ has bounded spheres separation.
\end{lemma}

Before proving this lemma, let us give few definitions and a fact that we will use in the proof.
We shall denote $(S_R)_\delta$ the $\delta$-neighbourhood of the sphere $S_R$.
Moreover, $A(R)$ (or simply $A$) denotes the annulus $B_{3R} - B_{2R}$.

\begin{definition}
Given $x\in S_R$ and $r>0$, let $\Sigma_x = \pi_R \ii (x)$ and $\Sigma_{x,r} = \pi_R \ii (B(x,r))$ be respectively the \textbf{shadow}, and the \textbf{$r$-shadow} of $x$ in $G$.

Let us write the \textbf{sector} $\Sigma^A_x = \Sigma_x \cap A$ and \textbf{$r$-sector} $\Sigma^A_{x,r} = \Sigma_{x,r} \cap A$.
Let finally $\Sigma_x(r') = \Sigma_x \cap (S_{r'})_\delta$ for each $r' \ge R$ and $\Sigma^A_D = \bigcup_{x\in D} \Sigma^A_x$ for each $D \subseteq S_R$ (see Figure \ref{figure: shadows and sectors}).
\end{definition}

\begin{figure}\caption{Shadows and sectors}\label{figure: shadows and sectors}
\begin{center}
\includegraphics[scale=1]{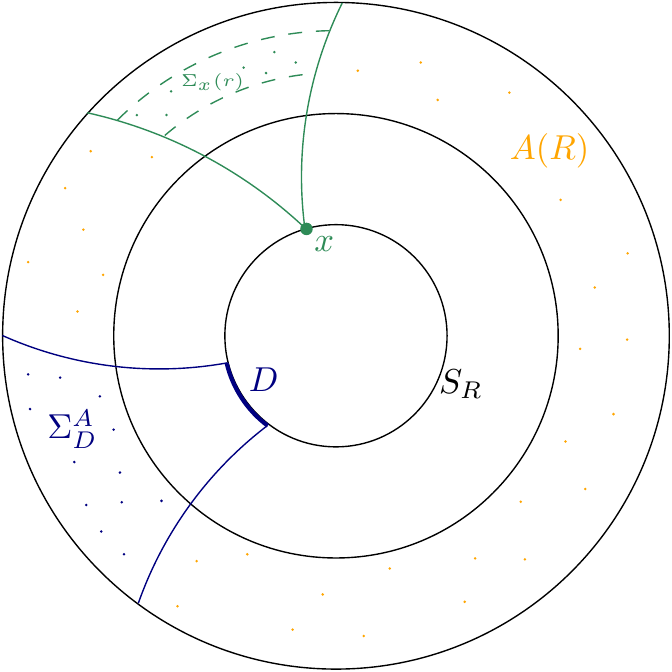}
\end{center}
\end{figure}

\begin{fact}\label{fact: size of sectors}

There are constants~$K,\alpha>0$ such that for all~$R \gg 0$ and $D\subseteq S_R$ we have
\[K^{-1} \alpha^R |D| \le |\Sigma^A_D| \le K \alpha^R |D|.\] 
In particular,
\[K^{-1} \alpha^R |S_R| \le |A(R)| \le K \alpha^R |S_R|.\]
\end{fact}

\begin{proof}
This follows from the proof of the Shadow Lemma, see Coornaert~\cite[Proposition 6.1]{coornaert1993mesures} (see also Gouëzel-Mathéus-Maucourant~\cite[Lemma 2.13]{gouezel2018entropy}). 
\end{proof}

\begin{proof}[Proof of Lemma \ref{lem: log sep to bdd sphere sep}]
It follows from Fact \ref{fact: size of sectors} that $|A(R)| \asymp \exp(R)$.
Let~$C \subset A(R)$ such that each connected component of~$A(R) - C$ contains at most~$\beta |A(R)|$ vertices, with $\beta = \left({4 |B_{8\delta}|^3 K^3}\right)^{-1}$, where $K$ is given by Fact \ref{fact: size of sectors}.
From the assumption that $\sep_G \preceq \log$, we can suppose~$|C|\preceq \log (|A(R)|) \preceq  R$.
Concretely, let $|C|\le cR$ for some~$c$.

Let $P\subset S_R$ be the set of all $x$ such that $\Sigma_{x,6\delta}(r) \cap C \ne \emptyset$ for all $2R \le r \le 3R$.
If $x\in P$ then $|\Sigma^A_x \cap C|\ge R/12\delta$.
Since the $6\delta$-sectors of pairs of vertices that are $12 \delta$ apart are disjoint, it follows that $|P| \le 12 \delta c |B_{12 \delta}|$.
So, there is a uniform bound on the size of $P_\delta$, the $\delta$-neighbourhood of $P$.
It remains to show that $P_\delta$ separates $(S_R)_\delta$.

\begin{claim}\label{claim: size of component in sector}
    There exists $K'>0$ such that for every $R \gg 0$, if $x\in S_R - P$ then $\Sigma^A_x-C$ has a component $T_x$ of size $|T_x| \ge \frac 1 {2K'} |\Sigma^A_x|$.
\end{claim}

\begin{proof}
    Since $x \notin P$, there exists $r\in [2R,3R]$ such that $\Sigma_{x,6\delta}(r) \cap C= \emptyset$.
    \begin{fact}\label{fact: sigma_x is (almost) connected}
    $\Sigma_{B_{2\delta}(x)}(r)$ is contained in a single connected component of $\Sigma^A_{x,6\delta} -C$
    \end{fact}
    \begin{proof}
	Given two points $y_1,y_2 \in \Sigma_{x,2\delta}(r)$, if $z$ is a vertex on a geodesic joining $y_1$ and $y_2$, then it is at distance at most $\delta$ from a point $z' \in [o,y_1] \cup [o,y_2]$. The $\delta$-slimness of the triangle given by $o, z$ and $z'$ implies that $\pi_R(z)$ is at distance at most $4\delta$ from $x$.
	Projecting the geodesic $[y_1,y_2]$ using $\pi_r$ gives a $2\delta$-step path joining $y_1$ and $y_2$ in $\Sigma_{x,4\delta} \cap S_r$.
	By proceeding to an interpolation process, we obtain a path joining $y_1$ and $y_2$ in $\Sigma_{x,6\delta} \cap (S_r)_\delta = \Sigma_{x,6\delta}(r)$. 
    Since $\Sigma_{x,6\delta}(r)$ do not meet $C$ and is included in $A$, this means exactly that $\Sigma_{x,2\delta}(r)$ is contained in a single connected component of $\Sigma^A_{x,6\delta} - C$. 
    \end{proof}
    Let $T_x$ be the component of $\Sigma^A_{x,6\delta} -C$ that contains $\Sigma_{B_{2\delta}(x)}(r)$.
    We shall consider shadows of vertices in $\Sigma_x(2R)$ and their intersection with $A$.
    In a similar way as in Fact \ref{fact: size of sectors}, we have a constant $K'$ such that any subset $Q \subset \Sigma_x(2R)$ satisfies 
   \begin{align}
	\label{equation: comparing Q and Sigma}
    (K')^{-1} \frac{\Sigma_Q^A}{\Sigma_x^A} \le \frac{|Q|}{\Sigma_x(2R)} \le K' \frac{\Sigma_Q^A}{\Sigma_x^A}  
    \end{align}
        Let $Q_x$ be the collection of all points $y\in \Sigma_{x,2\delta}(2R)$ such that $\Sigma_{y,4\delta}^A \cap C \ne \emptyset$.
    Since $|C| \le cR$ it follows that $|Q_x| \le c |B_{4\delta}| R \asymp R $.
    Since $|\Sigma_{B_{2\delta}(x)}(2R)| \asymp \exp(R)$, the complementary set $Q_x^c = \Sigma_{B_{2\delta}(x)}(2R) - Q_x$ satisfies
    \begin{align}
    	\label{equation: Q_x is small}
|Q_x^c| \ge \frac 1 2 |\Sigma_x(2R)|,
    \end{align}
 for any large enough $R$.
    
    For points $y \in Q_x^c$, in a similar manner as in Fact \ref{fact: sigma_x is (almost) connected}, $\Sigma_y^A$ is contained in a single connected component of $\Sigma^A_{x,6\delta}-C$ and intersects $\Sigma_{B_{2\delta}(x)}(r)$.
    Thus, $\Sigma_y^A \subseteq T_x$.
    Therefore we have $\Sigma_{Q_x^c}^A = \bigcup_{y \in Q_x^c} \Sigma_y^A \subset T_x$, which implies $|T_x| \ge \frac 1 {2K'} \Sigma_x^A$ from \eqref{equation: comparing Q and Sigma} and \eqref{equation: Q_x is small}. 
    \end{proof}

Let $D'\subset (S_R)_\delta-P_\delta$ be a connected subset.
We need to show that~$|D'| \le \frac {|(S_R)_\delta|} 2 $.
Let $D$ be the set of elements of $S_R$ that are at distance at most $\delta$ from $D'$.
Recall that $\Sigma^A_D = \bigcup_{x\in D} \Sigma^A_x$, and define similarly $T_D = \bigcup_{x\in D} T_x$, with $T_x$ given by Claim \ref{claim: size of component in sector}.

\begin{claim}\label{claim: components above components}
    $T_D$ is connected in $A(R)-C$.
\end{claim}

\begin{proof}
	Since $D_\delta$ is connected, it suffices to show that for any $x,x' \in D$ at distance at most $\delta$ from each other, $T_x$ and $T_{x'}$ are connected in $A(R) - C$.
	
    Let then $x,x' \in D$ be such that $d(x,x') \le \delta$.
    Then $\Sigma_{x,2\delta}(2R) \cap \Sigma_{x',2\delta}(2R)$ contains $\exp(R)$ points.
    Repeating the argument of the previous claim we see that, if $R$ is large enough, one of those points $y\in \Sigma_{x,2\delta}(2R) \cap \Sigma_{x',2\delta}(2R)$ is not in $Q_x$ nor in $Q_{x'}$.
    As above, $\Sigma_y^A$ is included in a single connected component of $A(R)-C$ and is also included in $T_x$ and $T_{x'}$.
    Thus $T_x$ and $T_{x'}$ are in the same connected component of $A(R)-C$.
\end{proof}

    By assumption on~$C$, this implies that we have 
\begin{align}
\label{eq: TD is small}
    |T_D| \le \beta |A(R)| = \frac{|A(R)|}{4|B_{8\delta}|^3 K^3}.
\end{align}   

\begin{claim}
We have
\begin{align}\label{eq: evaluating TD on Tx''s}
\sum_{x \in D} |T_x| 
\le |B_{8 \delta}||T_D|.
\end{align}    
\end{claim}

\begin{proof}
Let us show that there exists a map $\phi:D\to D$ such that 
\begin{itemize}
    \item $d(x,\phi(x)) \le 8\delta$ and in particular $|\phi\ii (\phi(x))|\le |B_{8\delta}|$,
    \item $|T_x| \le |T_{\phi(x)}|$, and
    \item if $y\ne y' \in \im \phi$ then $T_{y} \cap T_{y'} = \emptyset$.
\end{itemize} 

Assuming we have constructed such a map, the claim follows by
\begin{align}
\sum_{x \in D} |T_x| \le |B_{8 \delta}|\sum_{y \in \im \phi}|T_{y}|
\le |B_{8 \delta}||T_D|
\end{align}   

To construct the map $\phi$, let $x\in D$ be a point maximizing $|T_x|$. 

Let $Z\subseteq D$ be the collection of all points $x'\in D$ satisfying $T_{x'} \cap T_x \ne \emptyset$. 
Define $\phi$ on $Z$ by $\phi(x')=x$.
Note that each $x'\in Z$ is at distance $d(x,x')\le 8\delta$, and by the choice of $x$, $|T_{x'}| \le |T_x|$.

Remove all the points in $Z$ from $D$, and iterate the construction above until $\phi$ is defined on all $D$.
\end{proof}

    Without any loss of generality we can assume that we have $K>K'$.
	We deduce:
    \begin{align*}
|D'|  &\le |B_{8\delta}| |D|
\\&\le |B_{8\delta}| K \alpha^{-R} |\Sigma_D^A|\quad\text{from Fact \ref{fact: size of sectors}}
\\&\le |B_{8\delta}| K \alpha^{-R} \sum_{x \in D}|\Sigma_x^A|
\\&\le 2|B_{8\delta}| K^2 \alpha^{-R} \sum_{x \in D}|T_x|\quad\text{from Claim \ref{claim: size of component in sector}, assuming $K>K'$}
\\&\le 2|B_{8\delta}|^2K^2 \alpha^{-R} |T_D|\quad\text{from \eqref{eq: evaluating TD on Tx''s}}
\\&\le \frac 1 {2K|B_{8\delta}|} \alpha^{-R} |A(R)|\quad\text{from \eqref{eq: TD is small}}
\\&\le \frac 1 2 \frac{|S_R|}{|B_{8\delta|}}\quad\text{from Fact \ref{fact: size of sectors}}
\\&\le \frac 1 2 |(S_R)_\delta|
    \end{align*}
    This ends the proof of Lemma \ref{lem: log sep to bdd sphere sep}.
\end{proof}

\begin{definition}
    Let $X$ be a connected topological space. We say that a subset $F$ \textbf{topologically separates} $X$ if $X-F$ is not connected.
\end{definition}

\begin{lemma}\label{lem: bdd sphere sep to sep of boundary}
    If $G$ has bounded spheres separation, then $\partial G$ has a finite topologically separating set.
\end{lemma}
\begin{proof}
We start with the following claim.
\begin{claim}\label{claim: projecting separating sets P_R}
There exists $K>0$ such that for any $R<R'$, if $(P_{R'})_\delta$ separates $(S_{R'})_\delta$ into connected components of size at most $\frac 1 K |(S_{R'})_\delta|$, then $\left(\pi_R(P_{R'})\right)_{2\delta}$ separates $(S_R)_\delta$ into connected components of size at most $\frac 1 2 |(S_R)_\delta|$.
\end{claim}

\begin{proof}
If $x, y \in S_R - \left(\pi_R(P_{R'})\right)_{2\delta} $, then $\left(\pi_{R'}\ii(x_{2\delta})\right)_\delta$ and $\left(\pi_{R'}\ii(y_{2\delta})\right)_\delta$ are connected in $(S_{R'})_\delta - (P_{R'})_\delta$.
If moreover $d(x,y) \le \delta$, these two sets intersect.
This implies that for every connected subset $D$ of $(S_R)_\delta -  \left(\pi_R(P_{R'})\right)_{2\delta}$, the set $D':=(\pi_R\ii(D\cap S_R ))_\delta$ is connected in $(S_{R'})_\delta - (P_{R'})_\delta$.
The conclusion of claim follows from the fact that sizes of $D$ and $D'$ differ by some constant factor, as in Fact \ref{fact: size of sectors}. 
\end{proof}

From Lemma \ref{lem: log sep to bdd sphere sep}, for any large enough $R$ there exists $P_R \subset S_R$ of bounded size $M$ satisfying that $(P_R)_{2\delta}$ separates $(S_R)_\delta$ into connected components of size at most $\frac 1 2 |(S_R)_\delta|$.

\begin{claim}
Without any loss of generality, we can assume that we have $P_{R_1} = \pi_{R_1}(P_{R_2})$ for every $R_1 \le R_2$.
\end{claim}

\begin{proof}
We can assume without any loss of generality that the projection maps are chosen so that we have $\pi_{R_1}(x) = \pi_{R_1} \circ \pi_{R_2}(x)$ for every $R_1 < R_2$ and $x \in G - B_{R_2}$.

For every $R'>0$, Claim \ref{claim: projecting separating sets P_R} allows us to obtain $P_R$ for every $R < R'$ by projecting some subset of $S_{R'}$.
Then, if we have $R_1 < R_2 < R'$, with $P_{R_1}$ and $P_{R_2}$ obtained by projecting some subset of $S_{R'}$, we have $P_{R_1} = \pi_{R_1}(P_{R_2})$ since $\pi_{R_1} = \pi_{R_1} \circ \pi_{R_2}$.
Since the spheres of $G$ are finite, we can proceed to an extraction to obtain a sequence $R'_n$ such that for every $R>0$ the sequence $(\pi_{R}(P_{R'_n}))_{n \ge 0}$ is constant (it is only defined when $R'_n \ge R$). Without any loss of generality we can assume that we have $P_R = \pi_{R}(P_{R'_n})$.
We finally get the desired property that $P_{R_1} = \pi_{R_1}(P_{R_2})$ for every $R_1 < R_2$.

\end{proof}

Now the sequence $P_R$ has a limit $P \subseteq \partial G$ as $R\to\infty$.
To complete the proof of Lemma \ref{lem: bdd sphere sep to sep of boundary} it remains to show that $P$ topologically separates $\partial G$.


From above, there exist $\xi,\eta\in\partial G$ such that $\pi_R(\xi)$ and $\pi_R(\eta)$ are in different components of $(S_R)_\delta-(P_R)_\delta$ for all large enough $R$.
Assume for contradiction that $\xi$ and $\eta$ are in the same component of $\partial G- P$.
The boundary $\partial G$ is path connected, let $\gamma$ be a path in $\partial G - P$ connecting $\xi$ and $\eta$. 
There exists $\epsilon>0$ such that the path $\gamma$ avoids the $\epsilon$-neighbourhood of $P$ in $\partial G$.
Let $R$ be big enough so that $\pi_R\ii(x_{2 \delta}) \cap \partial G$ is of diameter $\le \epsilon/2$ for each $x\in S_R$.

Thus, $\pi_R\circ \gamma$ is a $2\delta$-step path in $S_R$ which avoids $(P_R)_{2\delta}$.
Completing it with a collections of geodesic arcs of length at most $\delta$, we get a path in $(S_R)_\delta$ which avoids $(P_R)_\delta$, and connects $\pi_R(\xi)$ and $\pi_R(\eta)$.
	A contradiction.
This ends the proof of Lemma \ref{lem: bdd sphere sep to sep of boundary}.
\end{proof}

We are now able to prove Theorem \ref{thm: main thm}.
\begin{proof}[Proof of Theorem \ref{thm: main thm}]
By Lemmas \ref{lem: log sep to bdd sphere sep} and  \ref{lem: bdd sphere sep to sep of boundary} we see that $\partial G$ is topologically separated by a finite set of points. It therefore has a local cut point. It follows from Bowditch \cite{bowditch1998cut} that $G$ splits over a virtually cyclic group or $G$ is a Fuchsian group.
\end{proof}
\section{Proof of Theorem \ref{theorem: gluing tori along filling curves}}\label{section: proof of theorem B}
In this section,  we construct a hyperbolic group with superlogarithmic separation profile whose boundary has conformal dimension one.
Let us start by giving the following proposition.

\begin{proposition}[{{\cite[Lemma 1.3]{benjamini2010separation}}}]
\label{prop: sep monotone under coarse embeddings}
Let $G$ and $H$ be bounded degree infinite graphs such that there exists a coarse embedding $G \to H$.
Then, $\sep_G \preceq \sep_H$.
\end{proposition}

Recall that quasi-isometric embeddings are examples of coarse embeddings.
This proposition implies that the separation profile is invariant under coarse equivalences and quasi-isometries.
In particular we can consider separation profiles more generally for metric spaces  that are coarsely equivalent to graphs of bounded degree.
This is what we will do in this section for the hyperbolic plane.\\

Let $\Sigma$ and $\Sigma'$ be two closed hyperbolic orientable surfaces,  and $\gamma \subset \Sigma$, $\gamma' \subset \Sigma'$ be two simple closed filling geodesic curves.
Recall that a curve on a surface is said to be \textit{filling} when its complementary is homeomorphic to a union of disks.
Let~$S = (\Sigma \sqcup \Sigma') / {\gamma \simeq \gamma'}$ be the space obtained by gluing~$\Sigma$ and~$\Sigma'$ along~$\gamma$ and~$\gamma'$.

The universal cover $\tilde S$ of $S$ consists of copies of hyperbolic planes, that we will call sheets, glued together along the geodesic lines which correspond to the lifts of~$\gamma$ and~$\gamma'$.

Let $F$ be one of the sheets covering $\Sigma$. For a lift $\tild \gamma$ of $\gamma$ let $F_{\tild \gamma}$ be the adjacent sheet covering $\Sigma'$ which is glued to $F$ along $\tild{\gamma}$.

Let~$R>0$.
Let $B_R$ (resp. $B_{R/3}$) be the balls of radius $R$ (resp. $R/3$) in $F$ centered around $o$.
Let us consider 
$$X = B_R \cup \bigcup_{\tild\gamma \cap B_{R/3} \ne \emptyset} B_{F_{\tild\gamma}}(o_{\tild\gamma},R/3) \subset \tilde S$$
where the union ranges over all lifts $\tild\gamma$ of $\gamma$ that intersect $B_{R/3}$ and $o_{\tild\gamma}$ is the closest point to $o$ on $\tild\gamma$. 
See Figure \ref{figure: the set X}.

\begin{proposition}\label{prop: cut bigger than R}
The set $X$ satisfies $\cut X \succ R.$
\end{proposition}

\begin{figure}\caption{the set $X$ of Proposition \ref{prop: cut bigger than R}}\label{figure: the set X}
\begin{center}
\includegraphics[scale=0.5]{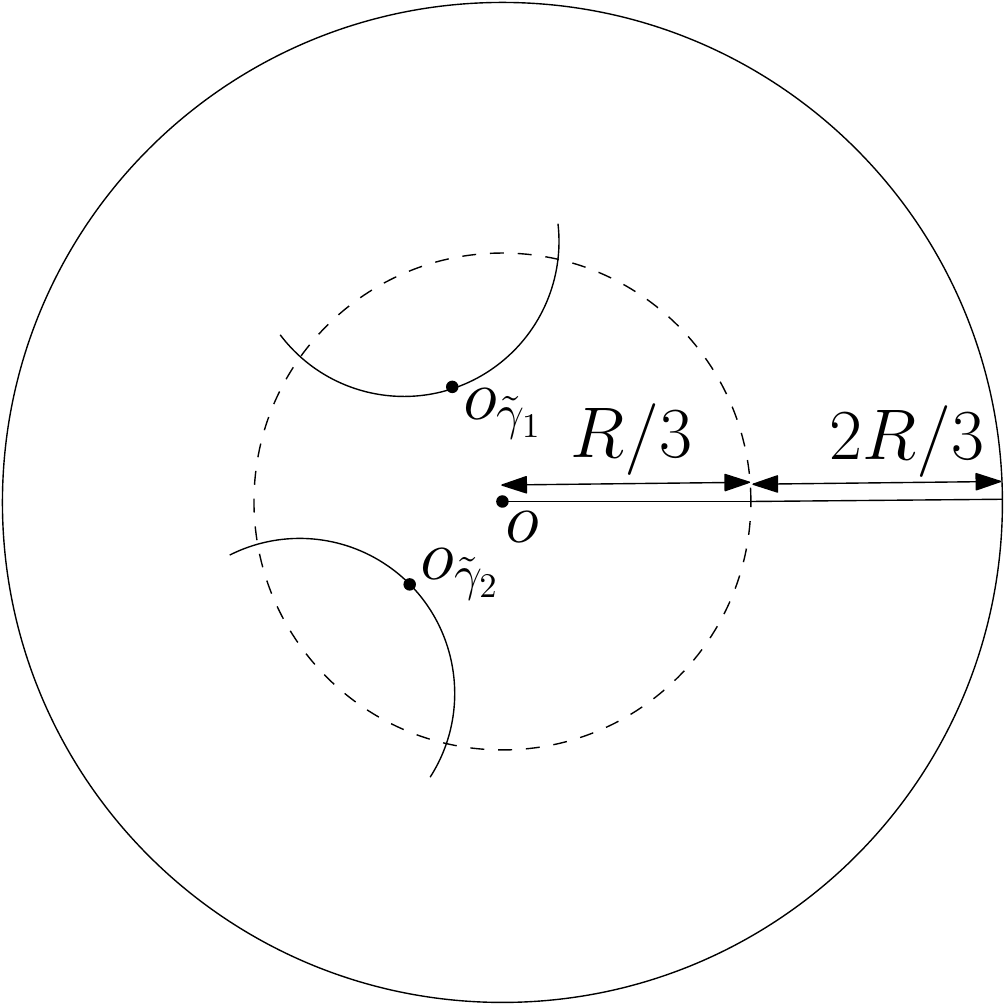}
\end{center}
\end{figure}

Let us prove how Theorem~\ref{theorem: gluing tori along filling curves} can be deduced from this proposition.

\begin{proof}[Proof of Theorem~\ref{theorem: gluing tori along filling curves}]

	The fundamental group~$\pi_1 S$ is quasi-isometric to the universal cover~$\tilde S$.
	Thus, we can compute the separation profile of~$\tilde S$ instead of that of~$\pi_1 S$, and the theorem follows from Proposition \ref{prop: cut bigger than R}.
\end{proof}

We now prove the proposition.

\begin{proof}[Proof of Proposition \ref{prop: cut bigger than R}]

\begin{claim}
Most of the volume of $X$ lies in the ball $B_R \subset F$: 
$$\vol (X) \asymp \vol(B_R).$$
\end{claim}

\begin{proof}
The volume of a ball $B(o,r)$ of radius $r$ in the hyperbolic plane is $$\vol(B(o,r)) = 2\pi (\cosh(r)-1)\asymp e^r.$$
The number of lifts of a geodesic that intersect a ball $B(o,r)$ is $\preceq e^r$.
Thus, $$\vol(B(o,R)) \preceq \vol (X) \preceq \vol(B(o,R)) + e^{R/3}\vol(B(o,R/3)) \preceq \vol(B(o,R)).\qedhere$$
\end{proof}

Let $C$ be a (1-thick) cutset of $X$, that is every connected component of $X - C$ has volume at most $\alpha\vol X $ for some $\alpha < 1$.
Up to taking a small enough $\alpha$, $C$ has to separate the ball~$B_R$.
We want to show that~$C$ must have volume strictly bigger than $\log \vol (X) \asymp \log \vol(B(o,R)) \asymp R$. 
Let us assume for a contradiction that we have (up to constants), $\vol C \preceq R$.

Let $\Lambda = \partial C$. 
The total length of $\Lambda$ is $O(R)$: indeed, $C$ has volume $R$ and can be chosen to be a union of balls of radius 1 in a given net and so the length of their boundary component has to be $O(R)$.

The components of $\Lambda$ are either proper arcs or simple closed curves in $B_R$.
Let $\hat\Lambda$ be the collection of geodesics with the same endpoints as the arcs of $\Lambda$.
Let $C'$ be the union of the components of $A - \hat\Lambda$ that include $C\cap \partial B_R$. See Figure \ref{figure: separating-sets-hyperbolic-ball}. 

\begin{figure}\caption{the separating set $C$ of the hyperbolic ball}\label{figure: separating-sets-hyperbolic-ball}
\begin{center}
\includegraphics[scale=0.7]{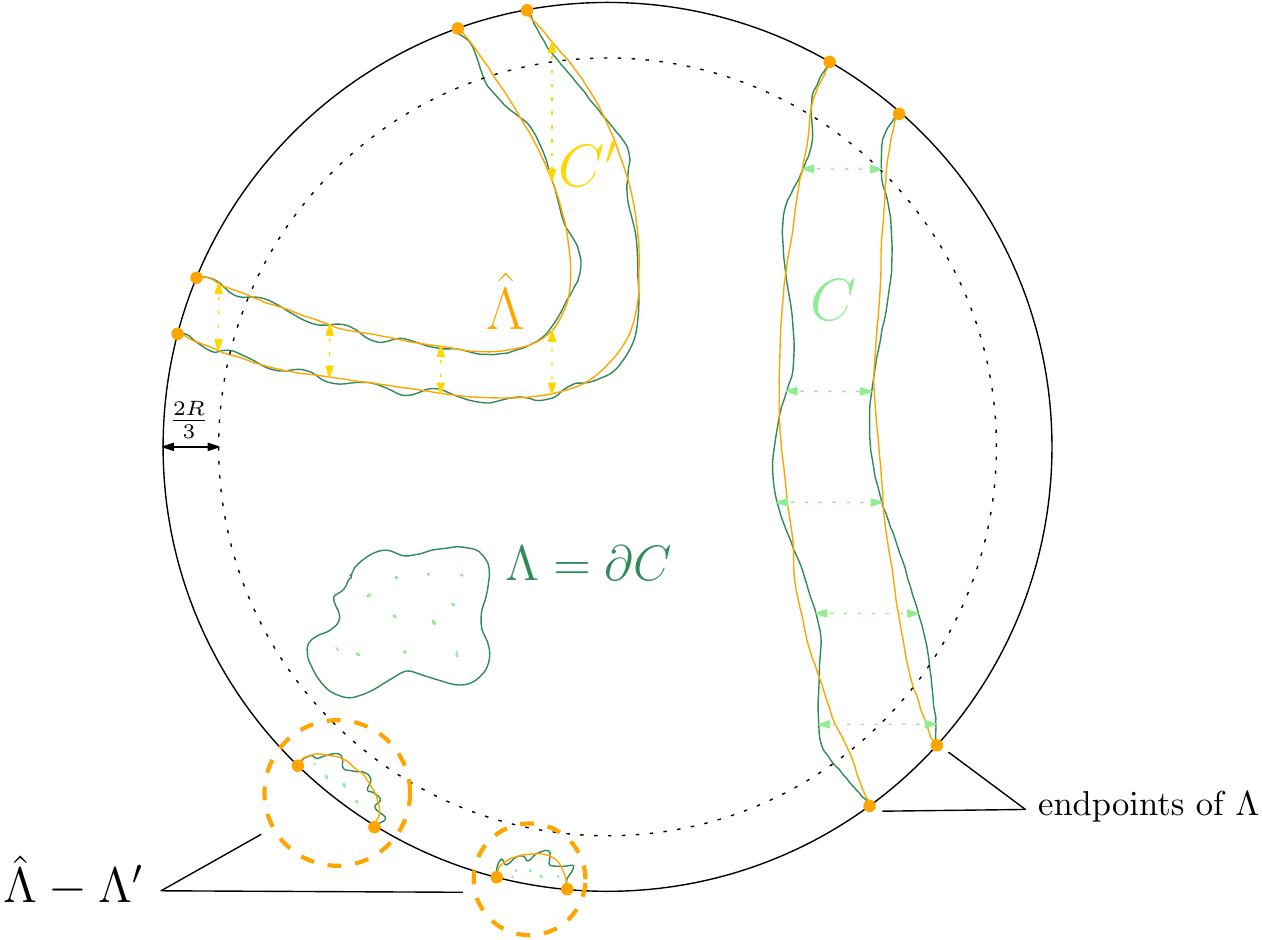}
\end{center}
\end{figure}

\begin{lemma}\label{lem: comparing C and C'}
\begin{enumerate}[label=(\roman*)]
    \item $\vol(C\triangle C')\preceq R$
    \item Every component $E$ of $B_R-C$ of size $\succ R$ corresponds to a unique component $E'$ of $B_R-C'$ such that $\vol(E\triangle E')\preceq R$, and vice versa.
\end{enumerate}
\end{lemma}
\begin{proof}
(i) The difference between the sets $C$ and $C'$ has boundary in $\Lambda \cup \hat \Lambda$.
Then, since the total length of $\Lambda$ and $\hat\Lambda$ is $O(R)$, it follows from the isoperimetric inequality on the hyperbolic plane, that $\vol(C \triangle C') \preceq \length(\Lambda\cup \hat \Lambda) \preceq R$.

(ii) By the isoperimetric inequality, a component $E$ of $B_R-C$ of size $\succ R$ must intersect $\partial B_R$. There is a component $E'$ of $B_R-C'$ with $E\cap \partial B_R = E'\cap \partial B_R$. The difference $E\triangle E'$ comprises of sets which are bounded by $\Lambda$ and $\hat\Lambda$. The volume of this difference can be bounded by $\preceq R$ again by the isoperimetric inequality.
\end{proof}

Let $\Lambda'$ be set of geodesics in $\hat\Lambda$ that meet $B_{R/3} = B_F(o,R/3)$.
Any geodesic in $\Lambda'$ must have a segment joining $\partial B_{R/3}$ and $\partial B_R$, and thus must have length at least $2R/3$. 
Since $\length(\hat \Lambda)\preceq R$ there are $O(1)$ many geodesics in $\Lambda'$.
Let $k$ be the number of geodesics in $\Lambda'$.

Let $m(x,y,z)$ denote the center of the geodesic triangle spanned by a triple points $x,y,z\in\bbH^2$. Let $$M= \{m(o,x,y) \mid x,y\in \Lambda'\cap \partial B_R\}.$$
There are at most $(2k)^2$ points in $M$.

Divide $B_{R/3}$ into $3((2k)^2+1)$ radial annuli of same width, called \emph{layers}.
By the Pigeonhole Principle, there exist three consecutive layers $A_-,A,A_+$ that do not contain a point of $M$.

\begin{claim}\label{claim: properties of walls}
For $R$ large enough, we have:
\begin{enumerate}[label=(\roman*)]
    \item The intersection $\Lambda'\cap A$ consists of geodesics joining the inner and outer boundaries of $A$.
    \item If $\alpha$ is a component of $\Lambda' \cap A$, and $\alpha$ is a subarc of $\lambda\in \Lambda'$, then $\alpha$ is at distance at most $\delta$ from one of the two geodesics connecting $o$ and $\partial \lambda$.
    \item If $\alpha_1,\alpha_2$ are component of $\Lambda' \cap A$, then either the Hausdorff distance $d_H(\alpha_1,\alpha_2)\le3\delta$ or they are at distance $\sqrt{R}$ apart.\footnote{The function $\sqrt{R}$ can be replaced by any function $o(R)$}
\end{enumerate}
\end{claim}

\begin{proof}
(i) Otherwise, a component $\alpha$ of $\Lambda'\cap A$ is a geodesic arc connecting the outer boundary of $A$ to itself. Let $p$ be the point on $\alpha$ closest to $o$, let $\lambda$ be the geodesic of $\Lambda'$ to which $\alpha$ belongs, and let $x,y$ be the endpoints of $\lambda$ in $S_F(o,R)$. Then, $m(o,x,y)\in M$ is at distance $\delta$ from $p$, contradicting the assumption that $A_-\cup A \cup A_+$ does not include points of $M$.

(ii) Consider the geodesic triangle consisting of the geodesic $\lambda$ and the two geodesics connecting its endpoints to $o$. By assumption, the center of this geodesic is not in $A''$, hence by slimness of triangles in $\bbH^2$ the segment $\alpha$ is $\delta$-close to one of the sides.

(iii) Let $\alpha_1,\alpha_2$ be components of $\Lambda' \cap A$. Let $\lambda_1$ (resp. $\lambda_2$) be the geodesic in $\Lambda$ containing $\alpha_1$ (resp. $\alpha_2$). By (ii), $\alpha_1$ (resp. $\alpha_2$) is $\delta$-close to a radial geodesic $\lambda_1'$ (resp. $\lambda_2'$) connecting $o$ and one of the endpoints of $\lambda_1$ (resp. $\lambda_2$). Let $\alpha_1' = \lambda_1'\cap A$ (resp. $\alpha_2' = \lambda_2'\cap A$).
It suffices to prove that $d_H(\alpha_1',\alpha_2')\le \delta$, or they are at distance $\sqrt{R}+2\delta$ apart.

Let $p_1,q_1 \in \alpha_1'$ (resp. $p_2,q_2\in \alpha_2'$) be the intersection of $\alpha_1'$ (resp. $\alpha_2'$) with the inner and outer boundaries of $A$ respectively.
If $d(q_1,q_2)\le \delta$ then $d_H(\alpha_1',\alpha_2')\le \delta$ by the convexity of the metric.
Similarly, if $d(p_1,p_2)\ge \sqrt{R}+2\delta$ then $\alpha_1',\alpha_2'$ are at least $\sqrt{R}+2\delta$ apart.
Otherwise, $d(p_1,p_2) \le \sqrt{R}+2\delta$ and $d(q_1,q_2) > \delta$, then the center of the triangle with sides $\lambda_1',\lambda_2'$ is at distance at most $\sqrt{R}+3\delta$ from $\alpha_1'$. For $R$ large enough, such a point must be in $A_-\cup A \cup A_+$ in contradiction to the assumption.
\end{proof}

From the claim above it follows that $\Lambda'\cap A$ consists of at most $2k$ geodesic segments connecting the inner and outer boundaries of $A$ and the relation defined by $\alpha_1 \sim \alpha_2$ if $d_H(\alpha_1,\alpha_2)\le 3\delta$ is an equivalence relation.
Let $\calW$ be a set of representatives of the classes of this relation. We call the elements in $\calW$  \emph{walls}.
We call the connected components of $A - \calW$ \emph{regions}.
By the claim, the walls bounding each region are at distance $\sqrt{R}$ apart.

\begin{claim}\label{claim: difference between region and component}
Let $D$ be a region in $A$, then there exists a (unique) component $E$ of $X-C$ such that $\vol (D - E) \preceq R$.
\end{claim}
\begin{proof}
Let $\alpha_1,\alpha_2$ be the walls bounding $D$. Let $\alpha_1',\alpha_2'$ be the inner most arcs in $D$ which belong to the equivalence classes of $\alpha_1,\alpha_2$ respectively. Let $E'$ be the connected component of $X-\hat\Lambda$ which includes the section of $A$ between $\alpha_1',\alpha_2'$. This section is contained in $D$, and $D - E'$ consists of two regions which are contained in the $3\delta$-neighborhood of $\alpha_1\cup \alpha_2$. Therefore, $\vol(D-E')\preceq R$.
The set $C'$ is bounded by geodesics in $\hat\Lambda$. It cannot contain the component $X-E'$, as otherwise $\vol(C')\succeq \vol(E') \succ R$.
Thus $E'$ is a component of $C'$.
By Lemma \ref{lem: comparing C and C'}, $E'$ corresponds to a unique component $E$ of $X-C$, and $\vol (D - E) \le \vol (D - E') +\vol (E' - E) \preceq R$.
\end{proof}

\begin{claim}
No component $E$ of $X-C$ contains more than $\frac{2}{3}$ of the layer $ A$
\end{claim}

\begin{proof}
Let $E$ be a component of $X-C$. Assume for contradiction that $\vol(E\cap A) > \frac{2}{3} \vol (A)$. For every $x\in E\cap A$ consider the ray $x^* = \bbR_{\ge 1} x \cap B_R = \{  tx \in B_R \mid t\ge 1\}$. 
Let $E_1 = \{x\in E \mid x^* \cap C = \emptyset\}$.
Since $\vol(C)$ consists of $O(R)$ balls of radius 1, the set $E-E_1$ consists of at most $O(R)$ 1-neighborhoods of arcs of length $O(R)$. Whence, $\vol(E-E_1) \preceq R^2$.
Consider the set $E_1^* = \bigcup_{x\in E_1} x^*$. 
Thus $\vol(E_1) > \frac{1}{2} \vol(A)$, and therefore also $\vol(E_1^*) > \frac{1}{2} \vol(B_R)$. 
The set $E_1^* \subseteq E$, thus $\vol (E) > \frac{1}{2} \vol(B_R)$.  We get a contradiction to the assumption that the volume of components of $X-C$ are at most $\frac{1}{2} \vol(B_R)$.
\end{proof}

By the previous two claims there are two regions $D_+,D_-$ of $A - \calW$ which correspond to two different components $E_+,E_-$ of $X-C$. We may assume that $D_1$ and $D_2$ are adjacent, and are separated by a wall $\alpha$.
Let $\alpha_m$ be the middle third subarc of $\alpha$.

\begin{claim}\label{claim: disjoint lifts}
There is $k\asymp R$, and disjoint lifts $\tild\gamma_1,\ldots,\tild\gamma_k$ of $\gamma$ such that $\tild \gamma_i \cap \alpha_m \ne \emptyset$.
\end{claim}

\begin{proof}
Consider the union $\Gamma = \bigcup \tild{\gamma}$ of all the lifts $\tild{\gamma}$ of $\gamma$ to the universal cover $F$ of $\Sigma$. 
Since $\gamma$ is filling in $\Sigma$, the connected components of $F-\Gamma$ are one of finitely many types of convex hyperbolic non-ideal polygons.
Let $d$ be the maximal diameter of these polygons. 
There exists an angle $\theta$ such that every geodesic line intersecting one of the polygons, forms an angle $\theta$ with at least one of its sides.
Let $\mu>0$ be such that if two geodesic lines $l_1,l_2$ in the hyperbolic plane intersect a third geodesic line $l$ at points of distance $\ge \mu$ and at angles $\ge \theta$, then $l_1,l_2$ do not meet.

Let $\ell = \length(\alpha_m)\asymp R$. 
Every segment of length $2d$ on $\alpha_m$ intersects a lift $\tild{\gamma}$ of $\gamma$ in an angle $\ge \theta$.
Thus, the geodesic segment $\alpha_m$ intersects at least $k=\ell / (\mu + 2d)$ lifts $\tild{\gamma}_1,\ldots,\tild{\gamma}_k$ of $\gamma$ in an angle $\ge \theta$.
Note that $k\asymp \ell \asymp R$.
By the choice of $\mu$, $\tild{\gamma}_1,\ldots,\tild{\gamma}_k$ are disjoint.
\end{proof}

Let $\tild\gamma_1,\ldots,\tild\gamma_k$ be the disjoint lifts as in Claim \ref{claim: disjoint lifts}.
The geodesic segments $\tild\gamma_i \cap D_\pm$ have length at least $\succeq \sqrt{R}$ by Claim \ref{claim: properties of walls} and by the choice of $\alpha_m$.
The circles around the point $\tild\gamma_i\cap \alpha_m$ in $F_{\tild \gamma_i}$ form $\Theta(\sqrt{R})$ disjoint paths connecting points in $D_+$ to points in $D_-$.
Considering these paths for all $\tild \gamma_i$, we get $\Theta(R^{3/2})$ disjoint paths connecting $D_+$ to $D_-$.
By Claim \ref{claim: difference between region and component}, $\vol(D_\pm - E_\pm) \preceq R$, and so we have at least $\succeq R^{3/2} - O(R)$ disjoint paths connecting $E_+$ to $E_-$.
Since $E_+$ and $E_-$ are different components of $X-C$, each of these paths meets $C$. We get $\vol(C) \succeq R^{3/2}$ which contradicts $\vol(C)\preceq R$.
This ends the proof of Proposition \ref{prop: cut bigger than R}.
\end{proof}

\bibliographystyle{plain}
\bibliography{biblio}
\end{document}